\documentclass[preprint,12pt]{elsarticle}

\usepackage{amsmath, amssymb, amsthm}
\usepackage{graphicx}
\usepackage{hyperref}
\usepackage{url}
\usepackage{enumerate}
\usepackage{mathrsfs}
\usepackage{tikz}
\usepackage{subcaption}

\usepackage{amsmath, amssymb, amsthm}
\usepackage{graphicx}
\usepackage{hyperref}
\usepackage{url}
\usepackage{enumerate}
\usepackage{mathrsfs}
\usepackage{tikz}
\usepackage{subcaption}
\usepackage{float}
\usepackage{enumitem}
\usetikzlibrary{positioning}
\usepackage{caption}
\usepackage{nicematrix,tikz}
\usepackage{tikz}
\usetikzlibrary{decorations.pathreplacing,positioning}
\newtheorem{theorem}{Theorem}[section]
\newtheorem{corollary}[theorem]{Corollary}

\newtheorem{lemma}[theorem]{Lemma}
\newtheorem{definition}[theorem]{Definition}
\newtheorem{remark}[theorem]{Remark}
\newcommand{\npos}[1]{n_{+}\!\left(#1\right)}
\newcommand{\nneg}[1]{n_{-}\!\left(#1\right)}
\newcommand{\nzero}[1]{n_{0}\!\left(#1\right)}
\journal{}

\begin{document}

\begin{frontmatter} 

\title{Distribution of signless Laplacian eigenvalues and
degree sequence}
\author[inst1]{S. Akbari}
\ead{s_akbari@sharif.edu}

\author[inst3]{M. Darougheh}
\ead{mostafadarougheh@gmail.com}

\author[inst2]{L. S. de Lima}
\ead{leonardo.delima@ufpr.br}

\author[inst2]{D. Traciná \corref{cor1}}
\ead{danieltracina@ufpr.br}

\cortext[cor1]{Autor correspondente}
\address[inst1]{Department of Mathematical Sciences, Sharif University of Technology}
\address[inst2]{Programa de Pós-Graduação em Matemática, Universidade Federal do Paraná, Curitiba, Brazil}
\address[inst3]{Department of Mathematics and Computer Science, University of Science and Technology}

\begin{abstract}
Let $G$ be a graph of order $n$ with degree sequence $d_1 \geq \cdots \geq d_{n}$. Let $m_{G} \emph{I}$ be the number of signless Laplacian eigenvalues in an interval $\emph{I}.$ In this paper, we characterize the distribution of the signless Laplacian eigenvalues in terms of the degree sequence of a graph within specific subintervals of $[0, \, 2n-2].$ We determine all graphs $G$ such that $m_{G}[d_n, 2n-2] \leq 2, \; m_{G}[d_{n-1}, 2n-2] = 1, \; m_{G}[0, d_1] \le 2.$ We also prove that there is no graph such that $m_{G}[0, d_3]=1$. In addition, we obtain all disconnected graphs such that $m_{G}[0, d_1] = 3$. Finally, we propose two open problems for future research.
\end{abstract}


\end{frontmatter}

\section{Introduction}
Let $G$ be a simple graph of order $n$, with the vertex set $V(G)$ and edge set $E(G)$. The degree sequence of $G$ is denoted by $d_1 \geq \cdots \geq d_n$. We write $A(G)$ for the adjacency matrix of $G$, and $D(G)$ for the diagonal matrix of vertex degrees. The matrix $Q(G) = A(G) + D(G)$ is called the \textit{signless Laplacian} of $G$, and its eigenvalues are denoted by $q_1(G) \geq \cdots \geq q_n(G)$. It is well known that the signless Laplacian matrix $Q(G)$ is positive semidefinite, so all its eigenvalues are non-negative. Furthermore, Gershgorin's circle theorem implies that $q_1(G) \le 2n-2$. Thus the signless Laplacian spectrum of any graph of order $n$ is contained in the interval $[0,\, 2n-2]$.

An extensive body of research has investigated how Laplacian eigenvalues are distributed within prescribed intervals; see, for example, \cite{ahanjideh, darougheh1, darougheh2, Cardoso2017, Guo2011, HEDETNIEMI201666, WANG2020307, XU202392}. These works establish connections between the placement of Laplacian eigenvalues and various structural characteristics of the underlying graph, including the domination number, independence number, vertex connectivity, diameter, and degree sequence. Recently, similar works have focused on the distribution of the eigenvalues of the signless Laplacian matrix within a given subinterval of $[0,\, 2n-2]$; see, for instance, \cite{xu2024distribution, ghodrati, belardocircunferencia, Lin2015}. For convenience, we denote by $m_GI$, where $I \subseteq [0,\, 2n-2]$, the number of signless Laplacian eigenvalues of a graph $G$ of order $n$ that lie in $I$ (counted with multiplicities). 

In this paper, we address the distribution of the signless Laplacian eigenvalues in terms of the degree sequence of a graph within specific subintervals of $[0, \, 2n-2].$ 

The paper is organized as follows. In Section~2, we state several known results from the literature that will be used in our proofs. In Section~3, we characterize all graphs satisfying $m_G[d_n,\, 2n-2] \le 2$, $m_G[d_{n-1},\, 2n-2] = 1$, and $m_G[0,\, d_1] \le 2$. For disconnected graphs, we also determine all graphs for which $m_G[0,\, d_1] = 3$. In addition, we prove that there is no graph $G$ such that $m_G[0,\, d_3] = 1$, along with several related auxiliary results. Section 4 contains our concluding remarks, in which we propose two problems.


\section{Preliminaries}

In this section, we introduce the notation and basic concepts used throughout the paper.
For a graph $G$ and a vertex $v \in V(G)$, the \emph{open neighborhood} of $v$ is
$N_G(v) = \{u \in V(G) : uv \in E(G)\}$, and the \emph{closed neighborhood} of $v$ is
$N_G[v] = N_G(v) \cup \{v\}$. The \emph{degree} of $v$ is $d_G(v) = |N_G(v)|$.
A set $S \subseteq V(G)$ is \emph{independent} if no two vertices of $S$ are adjacent.
The \emph{independence number} of $G$ is
$\alpha(G) = \max\{|S| : S \subseteq V(G)\ \text{is independent}\}$.

We denote by $P_n$, $C_n$, $S_n$, and $K_n$ the path, cycle, star, and complete graph
of order $n$, respectively. For a partition $n_1,\dots,n_k$ of $n$, we write
$K_{n_1,\dots,n_k}$ for the complete $k$-partite graph with part sizes $n_1,\dots,n_k$.

Given two graphs $G$ and $H$, their disjoint union is denoted by $G \cup H$.
If $G$ is a graph and $k \ge 1$, the notation $kG$ stands for the disjoint union of $k$ copies of $G$.
The \emph{join} of $G$ and $H$, denoted by $G \vee H$, is obtained from $G \cup H$ by adding all edges between $V(G)$ and $V(H)$.

For an $n \times n$ real symmetric matrix $M$, we write $ \operatorname{Spec}(M)$ for the multiset of its eigenvalues. We denote the multiplicity of an eigenvalue $\lambda$ of $M$ by ${\lambda ^{\left[ {{m}} \right]}}$, where $m$ is the multiplicity of $\lambda $.
In particular, the signless Laplacian spectrum of $G$ is
$ \operatorname{Spec}(G) = \operatorname{Spec}(Q(G))$ and the eigenvalues of \(A(G)\) will be denoted by
\(\lambda_1(G) \geq \cdots \geq \lambda_n(G)\). 

\begin{lemma}[\cite{Cvetkovic2010}, Theorem 1.3.15] \label{lema:Weyl}
Let \( A \) and \( B \) be \( n \times n \) Hermitian matrices. Suppose that \( \lambda_1(A) \geq \cdots \geq \lambda_n(A) \) and \( \lambda_1(B) \geq \cdots \geq \lambda_n(B) \) are eigenvalues of \( A \) and \( B \), respectively. Then
\[
\lambda_i (A + B) \leq \lambda_j (A) + \lambda_{i-j+1} (B) \quad (1 \leq j \leq i \leq n),
\]
\[
\lambda_i (A + B) \geq \lambda_j (A) + \lambda_{i-j+n} (B) \quad (1 \leq i \leq j \leq n).
\]
\end{lemma}

\begin{lemma}[\cite{Horn2012}, Theorem 4.3.28]\label{lema: principal submatrix}
If $M$ is a Hermitian matrix of order $n$ and $B$ is its principal submatrix of order $p$, then $\lambda_{n-p+i}(M) \leq \lambda_i(B) \leq \lambda_i(M) $ for $1 \leq i \leq p$.
\end{lemma}
The following fact is an immediate consequence of Weyl’s inequalities (Lemma~\ref{lema:Weyl}) and Cauchy’s interlacing theorem. We therefore omit the proof.

\begin{corollary}\label{cor: subgrafo induzido}
Let $G$ be a graph of order $n$, and let $H$ be an induced subgraph of $G$. Then, for every $a \in [0,\,2n-2]$,
\[
m_H[a,\,2n-2] \;\leq\; m_G[a,\,2n-2].
\] 
\end{corollary}

Given a graph \(G\) and a set of edges \(S \subseteq E(G)\), we denote by \(G - S\) the subgraph obtained from \(G\) by deleting all edges in \(S\). In particular, when \(S = \{e\}\) for some edge \(e\), we simply write \(G - e\). Conversely, if a graph \(G'\) satisfies \(G' = G - S\) for some \(S \subseteq E(G)\), then we may recover \(G\) from \(G'\) by adding back the edges of \(S\); in this case we write \(G = G' + S\).

\begin{lemma}[\cite{cvetkovic2007}, Theorem 2.1]\label{lem: menos_uma_aresta}
Let $G$ be a graph with $n$ vertices and let $e \in E(G)$. Then
\[
q_1(G) \geq q_1(G - e) \geq q_2(G) \geq \dots \geq q_n(G) \geq q_n(G - e).
\]
\end{lemma}

\begin{lemma}[\cite{Horn2012}, Corollary 4.3.5]\label{teo: inercia}
Let $A$ and $B$ be real symmetric $n\times n$ matrices. Suppose $\operatorname{rank}(B)=r$. Then the inertia of $A+B$ is controlled by that of $A$ in the following way:
\begin{align*}
\npos{A}-r \;\le\; \npos{A+B} \;\le\; \npos{A}+r,\\
\nneg{A}-r \;\le\; \nneg{A+B} \;\le\; \nneg{A}+r,\\
\nzero{A}-r \;\le\; \nzero{A+B} \;\le\; \nzero{A}+r,
\end{align*}
where for any symmetric matrix $X, \, \npos{X}, \nneg{X} \text{ and } \nzero{X}$ denote respectively the numbers of positive, negative and zero eigenvalues of $X$, counted with algebraic multiplicity.
\end{lemma}

\begin{lemma}[\cite{xu2024distribution}, Theorem 3.2]\label{teorema: Primeiro do Bo Zhou}
For a graph $G$ with $n$ vertices,\linebreak $\alpha(G) \leq \min\{m_G[d_n,\,2n-2],m_G[0,\,d_1]\}$.
\end{lemma}

\begin{definition}[\cite{Oboudi2017}, Definition 2.12]\label{def:Gs}
For every integer $n \geq 2$, let $G_n$ be the graph of order $n$ obtained from the 
disjoint union of complete graphs $K_{\lceil n/2\rceil}$ and $K_{\lfloor n/2\rfloor}$ 
as follows. Let 
$V(K_{\lceil n/2\rceil})=\{v_1,\dots,v_{\lceil n/2\rceil}\}$ and 
$V(K_{\lfloor n/2\rfloor})=\{w_1,\dots,w_{\lfloor n/2\rfloor}\}$. 
We then add edges between vertices of $K_{\lceil n/2\rceil}$ and $K_{\lfloor n/2\rfloor}$ 
so that the following conditions are satisfied:

\begin{enumerate}
\item $N[v_1]\subset \cdots \subset N[v_{\lceil n/2\rceil}]$ and 
$N[w_1]\subset \cdots \subset N[w_{\lfloor n/2\rfloor}]$.
\item $\lvert N[v_i]\cap V(K_{\lfloor n/2\rfloor})\rvert = i-1$ for $i=1,\dots,\lceil n/2\rceil$.
\item $\lvert N[w_j]\cap V(K_{\lceil n/2\rceil})\rvert =
\begin{cases}
j-1, & \text{if $n$ is even},\\[4pt]
j, & \text{if $n$ is odd},
\end{cases}
\quad\text{for } j=1,\dots,\lfloor n/2\rfloor.
$
\end{enumerate}
\end{definition}

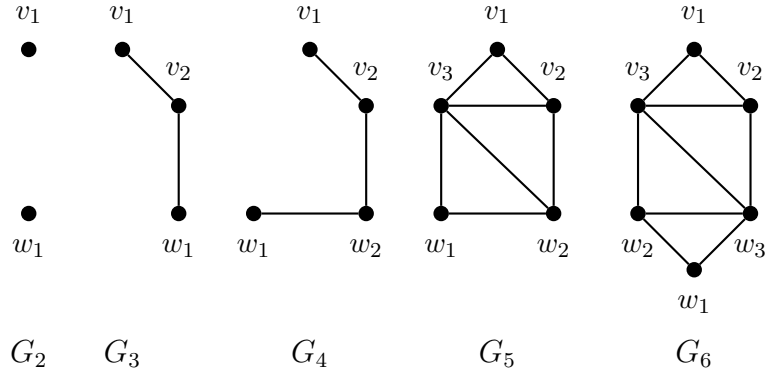
\begin{figure}[H]
\centering
\begin{tikzpicture}[
    scale=0.62,
    every node/.style={font=\normalsize},
    edge/.style={line width=0.8pt},
    vertex/.style={
        circle,
        draw,
        fill=black,
        line width=0.8pt,
        minimum size=5pt,
        inner sep=0pt
    },
    vlabel/.style={font=\small}
]

\node[vertex] (v1_2) at (-2,0) {};
\node[vlabel, above=0.1cm of v1_2] {$v_1$};

\node[vertex] (v3_2) at (-2,-3.5) {};
\node[vlabel, below=0.1cm of v3_2] {$w_1$};

\node[below=3.6cm of v1_2] {$G_2$};

\node[vertex] (v1_3) at (0,0) {};
\node[vlabel, above=0.1cm of v1_3] {$v_1$};

\node[vertex] (v2_3) at (1.2,-1.2) {};
\node[vlabel, above=0.1cm of v2_3] {$v_2$};

\node[vertex] (v3_3) at (1.2,-3.5) {};
\node[vlabel, below=0.1cm of v3_3] {$w_1$};

\draw[edge] (v1_3)--(v2_3);
\draw[edge] (v2_3)--(v3_3);

\node[below=3.6cm of v1_3] {$G_3$};

\node[vertex] (v1_4) at (4,0) {};
\node[vlabel, above=0.1cm of v1_4] {$v_1$};

\node[vertex] (v2_4) at (5.2,-1.2) {};
\node[vlabel, above=0.1cm of v2_4] {$v_2$};

\node[vertex] (v3_4) at (5.2,-3.5) {};
\node[vlabel, below=0.1cm of v3_4] {$w_2$};

\node[vertex] (v4_4) at (2.8,-3.5) {};
\node[vlabel, below=0.1cm of v4_4] {$w_1$};

\draw[edge] (v1_4)--(v2_4);
\draw[edge] (v2_4)--(v3_4);
\draw[edge] (v3_4)--(v4_4);

\node[below=3.6cm of v1_4] {$G_4$};

\node[vertex] (v1_5) at (8,0) {};
\node[vlabel, above=0.1cm of v1_5] {$v_1$};

\node[vertex] (v2_5) at (9.2,-1.2) {};
\node[vlabel, above=0.1cm of v2_5] {$v_2$};

\node[vertex] (v3_5) at (6.8,-1.2) {};
\node[vlabel, above=0.1cm of v3_5] {$v_3$};

\node[vertex] (v4_5) at (9.2,-3.5) {};
\node[vlabel, below=0.1cm of v4_5] {$w_2$};

\node[vertex] (v5_5) at (6.8,-3.5) {};
\node[vlabel, below=0.1cm of v5_5] {$w_1$};

\draw[edge] (v1_5)--(v2_5);
\draw[edge] (v1_5)--(v3_5);
\draw[edge] (v2_5)--(v3_5);
\draw[edge] (v2_5)--(v4_5);
\draw[edge] (v3_5)--(v4_5);
\draw[edge] (v3_5)--(v5_5);
\draw[edge] (v4_5)--(v5_5);

\node[below=3.6cm of v1_5] {$G_5$};

\node[vertex] (u1) at (12.2,0) {};
\node[vlabel, above=0.1cm of u1] {$v_1$};

\node[vertex] (u2) at (13.4,-1.2) {};
\node[vlabel, above=0.1cm of u2] {$v_2$};

\node[vertex] (u3) at (11,-1.2) {};
\node[vlabel, above=0.1cm of u3] {$v_3$};

\node[vertex] (u4) at (13.4,-3.5) {};
\node[vlabel, below=0.1cm of u4] {$w_3$};

\node[vertex] (u5) at (11,-3.5) {};
\node[vlabel, below=0.1cm of u5] {$w_2$};

\node[vertex] (u6) at (12.2,-4.7) {};
\node[vlabel, below=0.1cm of u6] {$w_1$};

\draw[edge] (u1)--(u2);
\draw[edge] (u1)--(u3);
\draw[edge] (u2)--(u3);
\draw[edge] (u2)--(u4);
\draw[edge] (u3)--(u4);
\draw[edge] (u3)--(u5);
\draw[edge] (u4)--(u5);
\draw[edge] (u4)--(u6);
\draw[edge] (u5)--(u6);

\node[below=3.6cm of u1] {$G_6$};

\end{tikzpicture}

\caption{Graph $G_s$, for $3 \leq s \leq 6$.}
\end{figure}

\begin{remark}[\cite{Oboudi2017}, Remark 2.13]\label{obs: recursao}
For every $n\ge 3$, $G_n$ is an induced subgraph of $G_{n+1}$. More precisely, if $n$ is even then $G_{n+1}\cong K_1\vee G_n$; if $n$ is odd, then $G_{n+1}$ is obtained from $G_n$ by adding a vertex $w$ adjacent to all vertices of $\{w_1,\dots,w_{\lfloor n/2\rfloor}\}$ (where $K_{\lfloor n/2\rfloor}$ is one of the parts of $G_n$).
\end{remark}

\begin{definition}[\cite{Oboudi2017}, Definition 2.17]
Let $G$ be a graph with vertex set $\{v_1,\dots,v_n\}$. We denote by $G[K_{t_1},\dots,K_{t_n}]$ the graph obtained by replacing each vertex $v_j$ with a complete graph $K_{t_j}$, $1\le j\le n$, and making every vertex of $K_{t_i}$ adjacent to every vertex of $K_{t_j}$ if and only if $v_i$ is adjacent to $v_j$ in $G$. For example, $K_2[K_p,K_q]\cong K_{p+q}$ and $K_2[K_p,K_q]\cong K_p\cup K_q$.
\end{definition}

\begin{lemma}[\cite{Oboudi2017}, Theorem 2.19]\label{teorema: Oboudi}
Let $G$ be a graph with eigenvalues\linebreak $\lambda_1(G) \geq \cdots \geq \lambda_n(G)$. Assume that $\lambda_3(G) < 0$. Then the following hold:

\begin{enumerate}
\item If $\lambda_1(G) > 0$ and $\lambda_2(G) > 0$, then $G \cong K_p \cup K_q$ for some integers $p, q \geq 2$ or there exist some positive integers $s$ and $t_1, \ldots, t_s$ so that $3 \leq s \leq 12$ and $t_1 + \cdots + t_s = n$ and $G \cong G_s[K_{t_1}, \ldots, K_{t_s}]$.
\item If $\lambda_1(G) > 0$ and $\lambda_2(G) = 0$, then $G \cong K_1 \cup K_{n-1}$ or $G \cong K_n - e$, where $e$ is an edge of $K_n$.
\item If $\lambda_1(G) > 0$ and $\lambda_2(G) < 0$, then $G \cong K_n$.
\end{enumerate} 
\end{lemma}

\begin{lemma}[\cite{Yan2002}]\label{lema: quase completo}
Let \(G = K_n- e\), where \(n\ge 3\). Then
\[
\operatorname{Spec}\big(G \big)
=
\left\{\,q_+^{[1]},\,q_-^{[1]},\, n-2^{[n-2]}\right\},
\]
where
\[
q_{\pm}=\frac{3n-6\pm\sqrt{(n-2)(n+6)}}{2}.
\]

\end{lemma}

\begin{lemma}[\cite{Das}, Corollary 3.2]\label{{q_2}{d_2} - 1}
Let $G$ be a graph with second maximum degree ${d_2}$. Then 
$${q_2}(G) \ge {d_2} - 1.$$
If ${q_2(G)} = {d_2} - 1$, then the maximum and the second maximum degree are adjacent and ${d_1} = {d_2}$.
\end{lemma}

\begin{lemma}[\cite{belardo}, Theorema 3.1]\label{lem: k3comd3}
Let \(G\) be a graph and $u$, $v$ and $w$ be vertices of degree at least $d_3$. We have:
\begin{enumerate}
\item if $u$, $v$ and $w$ induce $3K_1$, then $q_3(G) \ge d_3(G)$,
\item if $u$, $v$ and $w$ induce $K_3$ or ${P_2} \cup {K_1}$. then $q_3(G) \ge d_3(G) - 1$,
\item if $u$, $v$ and $w$ induce $P_3$, then $q_3(G) \ge d_3(G)- \sqrt 2 $.
\end{enumerate} 

\end{lemma}
\section{Main results}

\subsection{Distribution over the intervals $[d_n,\,2n-2]$ and $[d_{n-1}, 2n-2]$
} \label{subsec:31}

\vspace{0.3cm}

\begin{lemma} \label{corolario: grafo completo}
Let \( G \) be a graph of order $n$ with degree sequence\linebreak \( d_1 \geq \cdots \geq d_n \). Then, \( m_G [d_n,\, 2n-2] = 1 \) if and only if \( G = K_n \).
\end{lemma}

\begin{proof}
As it is well known (see, for instance,~\cite{Brouwer2012}), the signless Laplacian spectrum of the complete graph $K_n$ is given by 
\[
\{\,2(n-1)^{[1]},\; (n-2)^{[n-1]}\,\}.
\]
Hence, one side is clear for all $n \geq 1$. Now, suppose that $m_G[d_n, 2n-2] = 1$. By Theorem~\ref{teorema: Primeiro do Bo Zhou}, \( \alpha(G) \leq m_G[d_n, 2n-2] = 1 \). Thus \( \alpha(G) = 1 \), which implies that \( G \) is a complete graph, i.e., \( G = K_n \). 
\end{proof}

The following result can be viewed as a {direct consequence of [Theorem~1, ~\cite{LiZhou2017}].}

\begin{lemma}\label{lemma: spectrum of G_s}
Let $G = G_s[K_{t_1},\dots,K_{t_s}]$. For each block $B_i$ corresponding to $v_i \in V(G_s)$, denote by $\deg_i$ the degree of any vertex in $B_i$, $i=1,2,\dots,s$. Then
\[
\operatorname{Spec}(G) = \operatorname{Spec}(B) \;\cup\; \bigcup_{i=1}^s \{(\deg_i - 1)^{[t_i-1]}\},
\]
where $B$ is the $s \times s$ matrix defined by
\[
B_{ii} = \deg_i + (t_i - 1), \qquad
B_{ij} =
\begin{cases}
t_j & \text{if } v_i \sim v_j \text{ in } G_s,\\[2mm]
0 & \text{otherwise.}
\end{cases}
\]
\end{lemma}









\begin{lemma}\label{lema: g3 e g4}
Let $G = G_s[K_{t_1}, \ldots, K_{t_s}]$. Then, $m_G[d_n,\,2n-2] = 2$ if and only if either $G = G_3[K_t, K_1, K_t]$, with {$1 \leq t \leq 2$}, or $G = G_4[K_1, K_1, K_1, K_1]$.
\end{lemma}

\begin{proof} 
First, consider graphs $G^{\prime} = G_3[K_{t}, K_{1}, K_{t}]$ and $G^{\prime \prime} = G_4[K_{r}, K_{1}, K_{1}, K_{r}]$. By Lemma~\ref{lemma: spectrum of G_s}, part of the eigenvalues of $G^{\prime}$ and $G^{\prime \prime}$ come from their respective reduced matrices, say $B_3$ and $B_4$, which are given by:

\[
B_3 = \begin{pmatrix}
2t - 1 & 1 & 0 \\
t & 2t & t \\
0 & 1 & 2t - 1
\end{pmatrix},
\]
with spectrum $\operatorname{Spec}(B_3) = \left\{2t-1, \frac{4t - 1 \pm \sqrt{1 +8t}}{2}\right\}$, and 

\[
B_4 = \begin{pmatrix}
2r-1 & 1 & 0 & 0 \\
r & r+1 & 1 & 0 \\
0 & 1 & r+1 & r \\
0 & 0 & 1 & 2r-1
\end{pmatrix},
\]
with spectrum 
\begin{equation}\label{eq:B4}
\operatorname{Spec}(B_4) = \left\{2r, r-1, \frac{3r + 1 \pm \sqrt{r^2-2r +9}}{2}\right\}.
\end{equation}
As a consequence of Definition~\ref{def:Gs}, the minimum degree of $G^{\prime}$ is $t$ and the minimum degree of $G^{\prime \prime}$ is $r$. By their spectra, we obtain that $m_{G^{\prime}}[d_n,\, 2n - 2] = 2$ and $m_{G^{\prime \prime}} [d_n, 2n - 2] = 2$ if and only if $t \leq 2$ and $r = 1$. It completes the proof of one direction.

Now, let us assume that $m_G[d_n,\, 2n-2] = 2$ with $s = 3$, where $G = G_3[K_{t_1}, K_{t_2}, K_{t_3}]$, and define $t = \min\{t_1, t_3\}$.  If $t_2 \geq 3$, then any three vertices of $K_{t_2}$ have degree at least $d_3$. 
By Lemma~\ref{lem: k3comd3}, it follows that
\[
q_3(G) \geq d_3 - 1 \geq d_n,
\]
which is a contradiction. Therefore, we must have $t_2 \leq 2$. If $t_2 = 2$, we have $G = G_3[K_{t_1}, K_{2}, K_{t_3}]$, and $d_n(G) = \min\{t_1+1, t_3+1\}=t+1.$ Let $B'$ be the quotient matrix of the subgraph $G^{\prime} = G_3[K_t, K_2, K_t]$ obtained from Lemma~\ref{lemma: spectrum of G_s}:
\[
B^{\prime} =
\begin{pmatrix}
2t & 2 & 0 \\
t & 2t + 2 & t \\
0 & 2 & 2t
\end{pmatrix}.
\]
The spectrum of $B^{\prime}$ is given by 
\[
\operatorname{Spec}(B')
=
\left\{
2t,\;
2t + 1 - \sqrt{4t + 1},\;
2t + 1 + \sqrt{4t + 1}
\right\},
\]
and $B^{\prime}$ has at least two eigenvalues in the interval $[d_n,\, 2n-2]$. 
 By Lemma~\ref{lemma: spectrum of G_s}, $\deg_{2} - 1 = 2t \in \operatorname{Spec}(G')$ and $\deg_{2} - 1 \ge d_n$. It follows that $m_{G'}[d_n,\, 2n-2] \ge 3.$ By Corollary~\ref{cor: subgrafo induzido}, this implies that $m_{G}[d_n,\, 2n-2] \ge 3$, a contradiction.

Let $t_2=1$ and $t_{1} < t_{3}$. Observe that $H = K_{t_{3} + 1}$ is an induced subgraph of $G$, and $m_{H}[d_n,\, 2n-2] = t_{3} + 1$. Since $1 \leq t_1 < t_3$, we have that $t_3 \geq 2,$ which implies that $m_{H}[d_n,\, 2n-2] = t_{3} + 1 \geq 3.$ By Corollary~\ref{cor: subgrafo induzido}, this implies that
$m_{G}[d_n,\, 2n-2] \ge 3,$
contradicting the hypothesis. If $t = t_1 = t_3$, then $G = G_3[K_t, K_1, K_t]$, and $m_G[d_n,\, 2n-2] = 2$ if and only if $t \leq 2.$

{
Now, let $m_G[d_n, 2n-2] = 2$ and $s=4$, where $G = G_4[K_{t_1}, K_{t_2}, K_{t_3}, K_{t_4}]$. 
Let $t_2 \geq 2$. The degrees of four parts in the graph  $G = G_4[K_{t_1}, K_{t_2}, K_{t_3}, K_{t_4}]$, are ${t_1} + {t_2} - 1$, ${t_1} + {t_2} + {t_3} - 1$, ${t_2} + {t_3} + {t_4} - 1$ and ${t_3} + {t_4} - 1$. Clearly, ${d_n} \in \{ {t_1} + {t_2} - 1,{t_3} + {t_4} - 1\} $. If ${d_n} = {t_1} + {t_2} - 1$, then ${t_1} + {t_2} - 1 \le {t_3} + {t_4} - 1$. Thus ${t_1} + {t_2} \le {t_3} + {t_4}$. Now, we show that, $\min \{ {t_1} + {t_2} + {t_3} - 1,\,\,{t_2} + {t_3} + {t_4} - 1\}  \ge {d_n} + 1 = {t_1} + {t_2}$. Clearly, ${t_1} + {t_2} + {t_3} - 1 \ge {t_1} + {t_2}$. Since  $ {t_3} + {t_4} \ge {t_1} + {t_2}$, we have  $ {t_3} + {t_4} -{t_1}\ge   {t_2}\ge1$. So $ {t_2} + {t_3} + {t_4}-{t_1}\ge   {t_2}+1$ or equivalently ${t_2} + {t_3} + {t_4} - 1 \ge {t_1} + {t_2}$.  If ${d_n} = {t_3} + {t_4} - 1$, then similarly, one can see that, $\min \{ {t_1} + {t_2} + {t_3} - 1,\,\,{t_2} + {t_3} + {t_4} - 1\}  \ge {d_n} + 1$. These imply that $d_3  \geq d_n+1$. Let $d({v_i}) = {d_i}$, for $i = 1, \ldots ,n$ and consider $S = \{ {v_1},{v_2},{v_3}\} $. If the induced subgraph on $S$ is not isomorphic to $P_3$, then by Lemma \ref{lem: k3comd3}, $q_3(G) \geq d_3 - 1 \geq d_n$. Thus assume that this induced subgraph is $P_3$. Clearly, every induced subgraph of $G = G_4[K_{t_1}, K_{t_2}, K_{t_3}, K_{t_4}]$, isomorphic meets to $P_3$ at least one vertex of $K_{t_2}$. Now, since ${t_2} \ge 2$, one can choose three vertices of $G$, say $u$, $v$ and $w$ such that $d(u) \ge d(v) \ge d(w) \ge {d_n}$ and $uvw$ forms a triangle. Now, by Lemma~\ref{lem: k3comd3}, $q_3(G) \geq d_3 - 1 \geq d_n$, which is a contradiction. If $t_3 \geq 2,$ the same idea can be applied. So, we have that $t_2 = t_3 = 1,$ and consequently $G = G_4[K_{t_1}, K_1, K_1, K_{t_4}]$ with either $d_{n} = t_1$ or $d_{n} = t_4$. If $r = \min\{t_1, t_4\}$, $d_n = r$ and the graph $G^{\prime \prime} = G_4[K_{r}, K_1, K_1, K_{r}]$ is an induced subgraph of $G$. Taking $r\geq 2$, it is clear that $m_{G''}[d_n,\, 2n-2] \geq 3$ from the spectrum in \eqref{eq:B4}. By Corollary \ref{cor: subgrafo induzido}, $m_G[d_n,\, 2n-2] \geq 3$, which is a contradiction. With no lose of generality assume that ${t_1} \le {t_4}$.
Therefore, $G = G_4[K_1, K_1, K_1, K_{t_4}]$. By Lemma \ref{lemma: spectrum of G_s}, since $m_G[d_n,\, 2n-2] =2$, we have ${t_4} \le 3$. Clearly, $G = G_4[K_1, K_1, K_1, K_1]$.
}

Now, we show that if $s \geq 5$, then $m_G[d_n, 2n-2] \geq 3$, where $G \cong G_s[K_{t_1}, \ldots, K_{t_s}]$. Let $u$, $v$ and $w$ be vertices of $G$ whose degrees are at least $d_3$.  Note that those vertices cannot be the vertices of minimum degree by Definition ~\ref{def:Gs} since they do not have the minimum degree. Let $X$ denote the set of vertices of $G$ associated with the vertices of the clique $K_{\lceil n/2 \rceil}$, and let $Y$ denote the set of vertices of $G$ associated with the vertices of the clique $K_{\lfloor n/2 \rfloor}$. There are two possible distributions of the vertices $u$, $v$ and $w$. The first possibility is that all three vertices belong to the same partition, either $X$ or $Y$. In this case, $u$, $v$ and $w$ induce a $K_3$. By Lemma \ref{lem: k3comd3},
$q_3(G) \geq d_3 - 1 \geq d_n$, that is, $m_G[d_n,\, 2n-2] \geq 3$, which yields a contradiction.

The second possibility is that, without loss of generality, $u, v \in X$ and $w \in Y$. In this case, we may assume that $u$ and $v$ belong to the cliques associated with the vertices $v_{\lceil n/2 \rceil - 1}$ and $v_{\lceil n/2 \rceil}$, and that $w$ belongs to the clique associated with $w_{\lfloor n/2 \rfloor}$. Again, the vertices $u$, $v$ and $w$ induce a $K_3$, and the result follows from Lemma \ref{lem: k3comd3}.
The argument is analogous for the case where $u, v \in Y$ and $w \in X$.
\end{proof}

\begin{theorem}
Let \( G \) be a graph with degree sequence \( d_1 \geq \cdots \geq d_n \). Then \( m_G [d_n, 2n-2] = 2 \) if and only if {$
G \in \{ 2K_{\frac{n}{2}}, P_3, P_4, G_3[K_2, K_1, K_2] \}.
$}

\end{theorem}

\begin{proof}
If $
G \in \{ 2K_{\frac{n}{2}}, P_3, P_4, G_3[K_2, K_1, K_2]\}$ we can easily obtain that\linebreak \( m_G [d_n, 2n-2] = 2 \), and one side of the proof is complete. 

Now, assume that \( m_G [d_n, 2n-2] = 2 \). From Lemma~\ref{teorema: Primeiro do Bo Zhou} $\alpha(G) \leq 2$. Suppose that $\alpha(G) = 1$. In this case, $G \cong K_n$. By Lemma~\ref{corolario: grafo completo}, we have $m_G[d_n, 2n-2] = 1$, which is a contradiction. 

Consider that $\alpha(G) = 2$. If \( n \leq 2 \), we have \( G \in \{2K_1, P_2\} \). For \( n \geq 3 \), by using Lemma~\ref{lema:Weyl} with $i = 3$ and $j = n$, we get 
\[
\lambda_3(G) \leq q_3(G) - d_n,
\]
where $\lambda_i(G)$ is the $i$-th largest eigenvalue of the adjacency matrix $A(G).$
%
%
%
Since \( m_G [d_n, 2n-2] = 2 \), it follows that \( q_3(G) < d_n(G) \). Thus, \( \lambda_{3}(G) < 0 \), and by Lemma~\ref{teorema: Oboudi}, we have two remaining cases: $\lambda _2(G) = 0$ and $\lambda _2(G) > 0$. For the first case, $G \cong K_1 \cup K_{n-1}$ or $G \cong K_n - e$. If $G \cong K_1 \cup K_{n-1}$, then $d_n = 0$ and consequently $m_G[d_n,\,2n-2] \geq 3$, which is a contradiction. If $G \cong K_n - e$ and $n = 3$, it follows that $G \cong P_3$. If $n \geq 4$, then $d_n = n-2$ and by Lemma~\ref{lema: quase completo}, $m_G[d_n, 2n-2] \geq 3$, which is a contradiction. 

Let $\lambda_2 (G) > 0$. In this case, we have that $G \cong K_p \cup K_q$ for some integers $p, q \geq 2$ or there exist other positive integers $s$ and $t_1, \ldots, t_s$ such that\linebreak $3 \leq s \leq 12$, $t_1 + \cdots + t_s = n$ and $G \cong G_s[K_{t_1}, \ldots, K_{t_s}]$. Suppose that $G \cong K_p \cup K_q$ with $p \geq q$. In this case, if $p = q$, then $G \cong 2K_{\frac{n}{2}}$. Now, if $p \geq q + 1$, it follows that $m_G[d_n, 2n - 2] = p + 1 \geq 3,$ which is a contradiction.

Finally, assume that $G \cong G_s[K_{t_1}, \ldots, K_{t_s}]$. By means of Lemma~\ref{lema: g3 e g4}, the result follows, and the proof is complete.
\end{proof}

\begin{lemma}\label{e} 
Let $S$ be a set of vertices in a graph $G$ such that each vertex in $S$ has at least $e$ neighbors outside $S$. Let $\left| S \right| = m$, $m > 0$. Then ${q_m} \ge e$.
\end{lemma}
\begin{proof}
Let the principal submatrix ${Q_S}$ of $Q$ with rows and columns indexed by $S$. Let $Q(S)$ be the signless Laplacian of the subgraph induced on $S$. Then ${Q_S} = Q(S) + {D(S)}$, where ${D(S)}$ is the diagonal matrix such that $D{(S)_{ss}}$ is the number of neighbors of $s$ outside $S$. Since $Q(S)$ is positive semidefinite and $D(S) \ge eI$, all eigenvalues of ${Q_S}$ are not smaller than $e$. Thus by Lemma \ref{lema: principal submatrix}, ${q_m} \ge e$.
\end{proof}
\begin{lemma}\label{uniq vertex of min}
 Let $G$ be a graph of order $n$ with a unique vertex of minimum degree and $m_{G}[d_{n-1}, 2n-2] = 1$. If $S$ is a maximal independent set and $|S|= 2$, then the vertex with minimum degree is contained in $S$.
\end{lemma}

\begin{proof}
 Let ${v} \in V(G)$ and ${d_G}(v) = {d_n}$. Suppose that $v \notin S$. Then every vertex in $S$ has at least ${d_{n - 1}}$ neighbors outside $S$. Thus by Lemma \ref{e}, ${q_2}(G) \ge {d_{n - 1}}$, a contradiction.
\end{proof}

\begin{lemma}\label{alpha (G) - 1}
 If $G$ is a graph of order $n$ with exactly one vertex of minimum degree, then ${m_G}\left[ {{d_{n - 1}},\,2n - 2} \right] \ge \alpha (G) - 1$.
\end{lemma}

\begin{proof}
 Suppose $S = \left\{ {{v_1}, \ldots ,{v_{\alpha (G)}}} \right\}$ is an independent set. Since ${d_n} \ne {d_{n - 1}}$, $S$ contains at least $\alpha (G) - 1$ vertices of degree greater than or equal to ${d_{n - 1}}$. By Lemma \ref{e}, ${q_{\alpha (G) - 1}} \ge {d_{n - 1}}$. This implies that ${m_G}\left[ {{d_{n - 1}},2n - 2} \right] \ge \alpha (G) - 1$.
\end{proof}

We are ready to prove the next result that characterizes all graphs with $m_{G}[{d_{n-1}},\,2n-2] = 1$.

\begin{theorem}
Let $G$ be a graph of order $n$. Then ${m_G}[{d_{n - 1}},\,2n-2] = 1$ if and only if $G \in \left\{ {{K_n},\,{K_1} \cup {K_{n - 1}}} \right\}$.
\end{theorem}

\begin{proof}
If ${d_n} = {d_{n - 1}}$, then by Lemma \ref{corolario: grafo completo}, $G = {K_n}$. So suppose ${d_n} \ne {d_{n - 1}}$. By Lemma \ref{alpha (G) - 1}, $\alpha (G) \le 2$. If $\alpha (G) = 1$, then ${d_n} = {d_{n - 1}}$, a contradiction. Now, assume that $\alpha (G) = 2$. Let $v \in V(G)$ and ${d_G}(v) = {d_n}$. By Lemma \ref{uniq vertex of min}, any two vertices of the set $V(G)\backslash \{ v\} $ are adjacent. By Lemma \ref{{q_2}{d_2} - 1}, ${q_2(G)} \ge {d_2} - 1$. If ${d_2} - 1 \ge {d_{n - 1}}$, we have a contradiction. Thus ${d_2} - 1 < {d_{n - 1}}$, so we obtain 
$${d_2} = \cdots = {d_{n - 1}}.$$
Now, two cases occur:

 \noindent\textbf{Case 1:} If ${d_1} = {d_2}$, then $G = {K_1} \cup {K_{n - 1}}$. We find that ${m_G}\left[ {{d_{n - 1}},2n - 2} \right] = 1$.
 
 \noindent\textbf{Case 2:} If ${d_1} \ne {d_2}$, then $G$ is the graph obtained from ${K_{n - 1}}$ by joining one of its vertices to a new vertex. Thus $G = {G_3}\left[ {{K_1},{K_1},{K_{n - 2}}} \right]$. By Lemma \ref{lemma: spectrum of G_s}, some of the eigenvalues of $G$ come from its respective reduced matrix, denoted by ${B_3}$, which is given by:
\[
B_3 = \begin{pmatrix}
1 & 1 & 0 \\
1 & n-1 & n-2 \\
0 & 1 & 2n - 5
\end{pmatrix}.
\]

The characteristic polynomial of ${B_3}$ is 
$$\left( {x - n + 2} \right)\left( {{x^2} + (3 - 2n)x + 2n - 6} \right).$$
We have ${d_{n - 1}}(G) = n - 2$. It is easy to verify that this characteristic polynomial ${B_3}$ has $2$ roots greater than or equal to ${d_{n - 1}}(G)$, a contradiction. Hence, the proof is complete.
\end{proof}

\subsection{Distribution over the intervals $[0,\,d_1]$ and $[0,\,d_3]$
} \label{subsec:32}
Next, we prove that there is no graph of order $n \geq 3$ such that\linebreak \( m_G[0,\, d_1] = 1 \).

\begin{lemma}\label{lem: k_2}
Let \( G \) be a graph of order $n \geq 2$. Then \( m_G[0,\, d_1] = 1 \) if and only if \( G = K_2 \).
\end{lemma}

\begin{proof}
One side is clear. Now, we assume that $m_G[0,\, d_1] = 1$. By Lemma~\ref{teorema: Primeiro do Bo Zhou}, $\alpha(G) = 1$. As is well-known, the multiplicity of eigenvalue \( n - 2 \) of \( K_n \) is \( n - 1 \). By the hypothesis, it follows that \( n = 2 \). 
\end{proof}



\begin{theorem} \label{thm:mg2}
Let $G$ be a graph of order $n$. Then \( m_G [0,\, d_{1}] = 2 \) if and only if $G \in \{K_1 \cup K_2,\, P_3,\, K_3,\, 2K_1,\, 2K_2,\, C_5\}$.
\end{theorem}

\begin{proof}

If $G \in \{K_1 \cup K_2,\, P_3,\, K_3,\, 2K_1,\, 2K_2,\, C_5\}$ the proof follows easily. Assume that $m_G[0, d_1] = 2$. If $n = 2$, then $G \cong 2K_1$. If $n = 3$, then\linebreak $G \in \{K_1 \cup K_2, P_3, K_3\}$. If $n = 4$, then $G \cong 2K_2$. Now, let $n \geq 5$. By Lemma~\ref{teorema: Primeiro do Bo Zhou}, we have that $\alpha(G) \leq 2$. If $\alpha(G) = 1$, then $G \cong K_n$, which is a contradiction. For $\alpha(G) = 2$, let \( v \in V(G) \) such that \( d_G(v) = d_n \). Any pair of vertices in the set \( V(G) \setminus N_G[v] \) is adjacent. Next, we analyze the cases according to the cardinality of the clique whose vertices lie in $V(G)\setminus N[v]$.

 \noindent\textbf{Case 1:} \( |V(G) \setminus N_G[v]| = 1 \). 
Note that if \(d_n \leq 2\), then \(n \leq 4\), which we already analyzed. For \(d_n \geq 3\), note that $d_G(v) = n - 2$ and we know, by Lemma~\ref{lem: menos_uma_aresta}, that $q_2(G) \leq q_2(K_n) = n - 2$. Thus, $m_G[0,\,d_1] \geq 4$, a contradiction.

 \noindent\textbf{Case 2:} \( |V(G) \setminus N_G[v]| = 2 \). 
If \(d_n \leq 1\), then \(n \leq 4\), and that case was already analyzed. Let $d_n = 2$. Since $q_2(G) \leq q_2(K_5) = 3$, by Lemma~\ref{lem: menos_uma_aresta}, we have that $d_1 \leq 2$. Hence, $G \cong C_5$. If $d_n \geq 3$, it is clear that $d_1 \in \{n-2, n-3\}.$ If $d_1 = n-2,$ then $q_2(G) \leq n-2 = d_1,$ which implies $m_{G}[0,\,d_1] = n-1$, a contradiction. Assume that $d_{1} = n-3,$ and note that $d_{n}=n-3,$ and so $G$ is $(n-3)$-regular. 
For \( n = 6 \), the graph \( G \) is depicted in Figure~\ref{fig:n=6} and we have \( m_G[0, d_1] = 4 \), a contradiction. 
\begin{figure}[h]
 \centering
 \vspace{1em}
\begin{tikzpicture}[
    scale=1.4,
    every node/.style={font=\normalsize},
    edge/.style={line width=0.8pt},
    vertex/.style={
        circle,
        draw,
        fill=black,
        inner sep=0pt,
        minimum size=5pt,
        line width=0.8pt
    }
]

\begin{scope}[xshift=-0.5cm]

\node[vertex] (A1) at (-0.5,1) {};
\node[vertex] (A2) at (0.5,1) {};
\node[vertex] (A3) at (-1.1,0) {};
\node[vertex] (A4) at (1.1,0) {};
\node[vertex] (A5) at (0,-0.5) {};
\node[vertex] (A6) at (0,-1.4) {};

\draw[edge] (A1)--(A2);
\draw[edge] (A1)--(A3);
\draw[edge] (A1)--(A5);
\draw[edge] (A2)--(A4);
\draw[edge] (A2)--(A5);
\draw[edge] (A3)--(A4);
\draw[edge] (A3)--(A6);
\draw[edge] (A4)--(A6);
\draw[edge] (A5)--(A6);

\end{scope}

\end{tikzpicture}

 \caption{Graph $G$ for $n = 6$ in Case 2 of Theorem \ref{thm:mg2}.}
 \label{fig:n=6}
\end{figure}
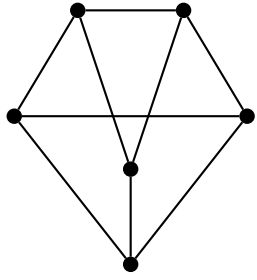


\noindent For \( n \geq 7 \), \( H_1 \) and \( H_2 \) graphs in Figure~\ref{fig:Teorema2-caso2-subgrafos} are possible induced subgraphs of $G$.
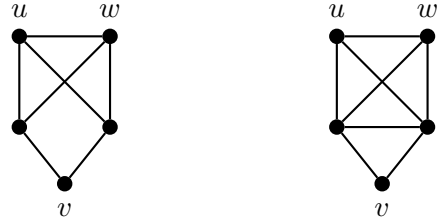
\begin{figure}[h]
 \centering
 \vspace{1em}
 \begin{tikzpicture}[scale=1.5,every node/.style={font=\large},line width=0.8pt]
\tikzset{vertex/.style={circle,draw,fill=black,inner sep=0pt,minimum size=5pt}}
\begin{scope}[xshift=-0.5cm]

\node[vertex] (A1) at (-0.8,0) {};
\node[vertex] (A2) at (0.0,0) {};
\node[vertex] (A3) at (-0.8,-0.8) {};
\node[vertex] (A4) at (0,-0.8) {};
\node[vertex] (A5) at (-0.4,-1.3) {};

\draw (A1)--(A2);
\draw (A1)--(A3);
\draw (A1)--(A4);
\draw (A2)--(A3);
\draw (A2)--(A4);
\draw (A3)--(A5);
\draw (A4)--(A5);

\node[font=\small] at (-0.8,0.23) {$u$};
\node[font=\small] at (0,0.23) {$w$};
\node[font=\small] at (-0.4,-1.53) {$v$};

\node[vertex] (B1) at (2,0) {};
\node[vertex] (B2) at (2.8,0) {};
\node[vertex] (B3) at (2,-0.8) {};
\node[vertex] (B4) at (2.8,-0.8) {};
\node[vertex] (B5) at (2.4,-1.3) {};

\draw (B1)--(B2);
\draw (B1)--(B3);
\draw (B1)--(B4);
\draw (B2)--(B3);
\draw (B2)--(B4);
\draw (B3)--(B5);
\draw (B4)--(B5);
\draw (B3)--(B4);

\node[font=\small] at (2,0.23) {$u$};
\node[font=\small] at (2.8,0.23) {$w$};
\node[font=\small] at (2.4,-1.53) {$v$};

\end{scope}

\end{tikzpicture}

\caption{Induced subgraphs of $G$ for $n \geq 7$ in Case 2 of Theorem \ref{thm:mg2}.}
\label{fig:Teorema2-caso2-subgrafos}
\end{figure}

\noindent The corresponding principal submatrices indexed by the vertices of $H_1$ and $H_2$ are $Q_{H_1}$ and $Q_{H_2}$, respectively. In both cases, we can decompose \( Q_{H_1} \) and \( Q_{H_2} \) as \( Q_{H_1} = (n-3)I + A_{H_1} \) and \( Q_{H_2} = (n-3)I + A_{H_2} \) where

\[
A_{H_1} = \begin{pmatrix}
0 & 0 & 1 & 1 & 1 \\[4pt]
0 & 0 & 0 & 1 & 1 \\[4pt]
1 & 0 & 0 & 1 & 1 \\[4pt]
1 & 1 & 1 & 0 & 0 \\[4pt]
1 & 1 & 1 & 0 & 0
\end{pmatrix} \qquad \text{ and } \qquad
A_{H_2} = \begin{pmatrix}
0 & 0 & 1 & 1 & 1 \\[4pt]
0 & 0 & 0 & 1 & 1 \\[4pt]
1 & 0 & 0 & 1 & 1 \\[4pt]
1 & 1 & 1 & 0 & 1 \\[4pt]
1 & 1 & 1 & 1 & 0
\end{pmatrix}. 
\]
Since \( A_{H_1} \) has three non-positive eigenvalues and \( A_{H_2} \) has three negative eigenvalues, it follows that the principal submatrices \( Q_{H_1} \) and \( Q_{H_2} \) have each three eigenvalues in the interval $[0, n-3]$. Therefore, by Lemma~\ref{lema: principal submatrix}, we conclude that $m_G[0, d_1] \geq 3$, a contradiction.

\textbf{Case 3:} \( |V(G) \setminus N_G[v]| \geq 3 \). Note that in this case, $d_n \geq 1$, otherwise $G \cong K_1 \cup K_{n-1}$ and consequently $m_G[0,\,d_1] = n - 1 \geq 3$. 
Thus, let us consider the principal submatrix $Q_H$ indexed by the vertex $v$ and the vertices $a, b, c \in V(G) \setminus N[v]$ given by
\[
Q_H = \begin{pmatrix}
d_n & 0 & 0 & 0 \\[4pt]
0 & d_a & 1 & 1 \\[4pt]
0 & 1 & d_b & 1 \\[4pt]
0 & 1 & 1 & d_c
\end{pmatrix}.
\]
We know that $1 \leq d_n \leq d_a \leq d_b \leq d_c$ and will show that $Q_H$ has at least three eigenvalues smaller than $d_c$. To this end, let us define \(u = (0, 1, 1, 1)^T\), $N = \operatorname{diag}(d_n - d_c,\ d_a - d_c - 1,\ d_b - d_c - 1,\ -1)$. It is easy to see that $N$ is negative semidefinite and $uu^T$ is positive semidefinite with rank $1$. Moreover, we can write
\[
Q_H - d_cI = N + u u^T.
\]
By Lemma~\ref{teo: inercia}, \(Q_H - d_c I\) can have at most $1$ positive eigenvalue. That is, $Q_H$ has at least three eigenvalues which does not exceed $d_c$. Therefore, by Lemma~\ref{lema: principal submatrix}, we can conclude that $m_G[0,\,d_1] \geq 3$, a contradiction. 
\end{proof}


\begin{corollary}
Let $G = G_1 \cup G_2$ be a disconnected graph of order $n$. Then \( m_G [0, d_{1}] = 3 \) if and only if one of the following holds: 
\begin{enumerate}[label=\roman*)]
\item $G_1 \cong K_1$ and $G_2 \in \{K_1 \cup K_2,\, P_3,\, K_3,\, 2K_1,\, 2K_2,\, C_5\}$, or 
\item $G_1 \cong K_2$ and $G_2 \in \{K_1 \cup K_2,\, 2K_1,\, 2K_2\}$. 
\end{enumerate}
\end{corollary}
\begin{proof}
By hypothesis, $m_G[0,\,d_1] = 3$, which implies that $m_{G_1}[0,\, d_1] + m_{G_2}[0,\, d_1] = 3$. As is well known, we have $ q_n(G_1) \in [0,\, d_1(G_1)]$, and, consequently, $ m_{G_1}[0,\, d_1] \geq 1$. Thus, without loss of generality, we may assume that $m_{G_1}[0,\, d_1] = 1$. Applying Lemma~\ref{lem: k_2}, we conclude that $G_1 \cong K_1$ or $G_1 \cong K_2$. If \(G_1 \cong K_1\), then all possible graphs for \(G_2\) are listed in Theorem~\ref{thm:mg2}, since in this case \(d_1(G_2) = d_1\). If \(G_1 \cong K_2\), by hypothesis, \(m_{G_1}[0, d_1] = 1\), which implies that \(d_1 \leq 1\). By Theorem~\ref{thm:mg2}, if \(d_1(G_2) = 0\), it follows that \(G_2 \cong 2K_1\), and if \(d_1(G_2) = 1\), it follows that
$G_2 \in \{ K_1 \cup K_2,\, 2K_2 \}$.
\end{proof}

\begin{theorem}
There is no graph of order $n \ge 3$ such that ${m_G}[0,{d_3}] = 1$.
\end{theorem}

\begin{proof}
By computer computation we checked the assertion for all graphs up to $5$ vertices.
Thus assume that $n\ge 6$. For each vertex $v_i\in V(G)$ let $d_G(v_i)=d_i$ for $1\le i\le n$.
By the pigeonhole principle there exist at least two vertices of degree at most $d_3$ such that either both are neighbors of $v_3$ or both are non-neighbors of $v_3$. Without loss of generality assume that these vertices are $v_4$ and $v_5$.

\noindent Define vectors $x=(x_1,\dots,x_n)^T$ and $y=(y_1,\dots,y_n)^T\in\mathbb{R}^n$ by
\[
 \begin{aligned}
 x_i=
 \begin{cases}
 {1 \over \sqrt{2}} & i = 4,\\
 {-1 \over \sqrt{2}} & i = 5,\\
 0 & \text{otherwise},
 \end{cases}
 \qquad
 y_i=
 \begin{cases}
 1 & i = 3,\\
 0 & \text{otherwise}.
 \end{cases}
 \end{aligned}
 \]
Let $S=\operatorname{span}\{x,y\}$ and let $z\in S$ with $\|z\|=1$.
By the Courant--Fischer theorem,
\[
q_{n-1}\le \max_{\substack{z\in S\\ \|z\|=1}} z^T Q z.
\]
Write $z=a x + b y$ with $a,b\in\mathbb{R}$; since $x\perp y$ and $\|z\|=1$ we have $a^2+b^2=1$.
Then
\[
z^T Q z = a^2\, x^T Q x + b^2\, y^T Q y + 2ab\, x^T Q y.
\]

\noindent Clearly, we get
\[
y^T Q y = d_3.
\]
Notice that ${x^T}Qy = \sum\limits_{1 \le i,j \le n} {{x_i}{q_{ij}}{y_j}} $. Since $v_3$ is either adjacent to both $v_4$ and $v_5$, or to neither, so we obtain that ${x^T}Qy = 0.$

\noindent In addition,
\[
 \begin{aligned}
 x^T Q x &=\sum_{v_i v_j\in E(G)} (x_i+x_j)^2=
 \begin{cases}
 {{{{d_4} + {d_5}} \over 2} - 1} & v_4v_5\in E(G),\\
 {{{{d_4} + {d_5}} \over 2}} & v_4v_5\notin E(G).\\
 \end{cases}
 \end{aligned}
\]

\noindent Therefore
\[
z^T Q z = a^2 \alpha + b^2 d_3,
\]
where
\[
\alpha =
\begin{cases}
{{{{d_4} + {d_5}} \over 2} - 1} & v_4v_5\in E(G),\\
 {{{{d_4} + {d_5}} \over 2}} & v_4v_5\notin E(G).\\
\end{cases}
\]
Since $a^2+b^2=1$ and $\alpha \le {d_3}$, the maximum of $a^2 \alpha + b^2 d_3$ occurs when $a=0$ and $b=1$.
Thus ${q_{n - 1}} \le {d_3}\ $, and the proof is complete.
\end{proof}

\section{Concluding remarks}

In Subsection \ref{subsec:31}, we characterized all graphs of order $n$ such that\linebreak \( m_G [d_{n}, 2n-2] \le 2 \). We propose the following problem.

\vspace{0.4cm}

\noindent \textbf{Problem 1.} Characterize all graphs $G$ of order $n$ such that\linebreak \( m_G [d_{n}, 2n-2] = 3 \).

\vspace{0.4cm}

In Subsection \ref{subsec:32}, we proved a characterization of disconnected graphs such that \( m_G [0, d_{1}] = 3 \). Here, we propose the following.

\vspace{0.4cm}

\noindent \textbf{Problem 2.} Characterize all connected graphs $G$ of order $n$ such that \( m_G [0, d_{1}] = 3 \).

\section*{Acknowledgments}
The research of Daniel Traciná was partially financed by the Coordenação de Aperfeiçoamento de Pessoal de Nível Superior – Brasil (CAPES) – Finance Code 001.

The research of Leonardo de Lima is supported by CNPq grant 305988/2025-5.\\
\bibliographystyle{elsarticle-num}
\bibliography{referencias}

\end{document}